\documentclass[10pt]{article}       
\usepackage{geometry}               
\usepackage{graphicx}
\usepackage{caption}
\usepackage{subcaption}
\usepackage{amssymb}
\usepackage{amsthm}
\usepackage{amsmath}
\usepackage{bm}
\usepackage{enumerate}
\usepackage{lipsum}
\usepackage[linesnumbered, ruled]{algorithm2e}

\newtheorem{theorem}{Theorem}
\newtheorem{proposition}{Proposition}
\newtheorem{lemma}{Lemma}

\newtheorem{definition}{Definition}

\title{\LARGE \bf
Identification of Parameters and Initial Values for Reaction-Diffusion Systems in Protein Networks (Extended Version)}

\author{Insoon Yang 
\thanks{I. Yang is with the Department of Electrical Engineering and Computer Sciences, University of California, Berkeley, CA 94720, USA
        {\tt\small iyang@eecs.berkeley.edu}}%
\and Claire J. Tomlin
\thanks{C. J. Tomlin is with the Department of Electrical Engineering and Computer Sciences, University of California, Berkeley, CA 94720, USA and the Life Sciences Division, Lawrence Berkeley National Laboratory, Berkeley, CA 94720, USA
        {\tt\small tomlin@eecs.berkeley.edu}}%
}

\date{}

\begin{document}

\maketitle
\pagestyle{plain}

\begin{abstract}
Spatio-temporal biochemical signaling in a large class of protein-protein interaction networks is well modeled by a reaction-diffusion system. The global existence of the solution to the reaction-diffusion system is determined by the reaction kinetics model and the protein network topology. We propose a novel reaction kinetics model that guarantees that the reaction-diffusion system with this model has a nonnegative invariant global classical solution for any network topology.
We then present a computational method to identify the unknown parameters and initial values for a reaction-diffusion system with this reaction kinetics model. 
The identification approach solves an optimization problem that minimizes the cost function defined as the $L^2$-norm of the difference between the data and the solution of the reaction-diffusion system. 
We utilize an adjoint-based optimal control method to obtain the gradients of the cost function with respect to the parameters and initial values. 
The regularity of the global classical solutions of the reaction-diffusion system and its corresponding adjoint system avoids situations in which the gradients blow up, and therefore guarantees the success of the identification method for any network structure.
Utilizing this gradient information, an efficient algorithm to solve the optimization problem is proposed and applied to estimate the mass diffusivities, rate constants and initial values of a reaction-diffusion system that models protein-protein interactions in a signaling network that regulates the actin cytoskeleton in a malignant breast cell.
\end{abstract}

\section{Introduction}
Reaction-diffusion systems have been widely used as fundamental models for the spatio-temporal dynamics of biochemical concentrations in complex protein networks \cite{Murray2003}.
Either data from new experiments or data from the literature can be used to directly determine the parameters of these reaction-diffusion systems, such as the mass diffusivities, rate constants, and initial values. However, the number and types of parameters that can be obtained via these sources are limited.
Although these parameters have physical meanings, 
the estimates of the model parameters solely based on physical laws often give ranges at best.
The lower and upper bounds of these ranges can vary by many orders of magnitude. 
Furthermore, the system may not explain the experimental data, even when all of the parameters are within their respective ranges.
Thus, a method that finds the set of parameters and initial values within a physically reasonable range that best matches the reaction-diffusion system with the experimental data is of considerable interest.
To computationally identify the parameters and initial values, we pose an optimization problem whose objective is to minimize the difference between the solution of the reaction-diffusion system and the data.

Several optimization-based parameter identification methods for reaction-diffusion partial differential equations (PDEs) have been developed in the more general context of parabolic equations. $(i)$ \emph{Semi-discrete} methods pose an approximate optimization problem by approximating a parabolic equation with a system of ordinary differential equations (ODEs)\cite{Banks1988, Ackleh1996}. 
However, the appropriate spatial discretization scheme for which the solution of the adjoint system (the dual of the ODE system)  
 converges to that of the adjoint equation of the parabolic equation is difficult to select \cite{Banks1992}.
$(ii)$ \emph{Discretize-then-optimize} methods fully discretize a weak form of the problem in time and space and then optimize the discretized problem \cite{Meidner2007}.
$(iii)$ \emph{Optimize-then-discretize} methods first obtain an analytic form of the gradient of the cost function with respect to the parameters by utilizing weak formulations of the state and adjoint equations and then discretize the problem to numerically solve the optimization problem \cite{Kravaris1985}. However,
when the weak solution of the reaction diffusion system blows up in finite time and so does that of the adjoint system, neither $(ii)$ nor $(iii)$ is able to compute the gradient of the cost function with respect to the initial values. In this case, neither framework is able to identify the initial values. 
The existing parameter identification methods may fail for some protein network topology, 
since the blow up property of reaction-diffusion systems is related to the network connectivity \cite{Pierre2000}. 
Because protein network structures of interest are diverse and complicated, an identification approach with guaranteed success for any network topology is highly desired.

In this article, we propose a novel reaction kinetics model such that the reaction-diffusion system with this model has a global classical solution regardless of the protein network topology. The reaction kinetics model has two key advantages. 
First, a reaction-diffusion system that implements this reaction kinetics model is an adequate modeling framework for general protein-protein interactions because the solution is nonnegative invariant and does not blow up in finite time.
Second, regardless of the protein network topologicogy, we have well-defined and bounded gradients of the cost function
with respect to the mass diffusivities, rate constants, and initial values
if we employ the reaction kinetics model.
With an analytic formula for the gradients based on an adjoint system,
we are able to efficiently solve the identification problem by simultaneously optimizing all unknown parameters and initial values of the system.
The boundedness of the gradients enhances the robustness of the optimization algorithms by preventing potential failure of the adjoint-based optimal control method: 
if the gradients tend to infinity, the algorithms might be terminated before finding an optimum. Thus, for any network topology, the reaction kinetics model that we propose guarantees the \emph{well-posedness} of the adjoint-based optimal control technique for the identification of reaction-diffusion systems.  

\section{Reaction-Diffusion Systems in Protein Networks}

Assume that the domain $\Omega$ is an open, bounded and connected subset of $\mathbb{R}^\eta$ with the boundary $\partial \Omega$ and outer normal vector $\nu$. We consider the following reaction-diffusion system to model the spatio-temporal dynamics of the biochemical concentrations (or densities) in a protein network: for $i=1, \cdots, N$,
\begin{subequations} \label{main_eq}
\begin{align}
& \frac{\partial u_i}{\partial t} - d_i \Delta u_i = r_i(u, k) \quad \mbox{in } \Omega \times (0,T) \label{pde} \\
&\frac{\partial u_i}{\partial \nu} = 0 \quad \mbox{on }\partial \Omega \times (0,T) \label{nbc}\\
&u_i(x,0)=u_i^0(x) \quad \mbox{in } \Omega \times \{t=0\}, \label{init}
\end{align}
\end{subequations}
where $u := u(x,t) = (u_1(x,t),\cdots,u_N(x,t))$ are the  concentration levels of $N$ proteins, $d=(d_1, \cdots, d_N) \in (0,+\infty)^N$ are the mass diffusivities,
and $k=(k_1,\cdots,k_M) \in (0,+\infty)^M$ are the rate constants.
Note that \eqref{nbc} and \eqref{init} specify the Neumann boundary conditions and initial conditions, respectively.
Assume that the initial value $u^0$ is in $L^{\infty}(\Omega)^N$ and $u^0(x) >0$ for all $x \in \Omega$. 
We call $r_i$ the \emph{reaction function} of the $i$th protein. The structure of the reaction function is determined by two factors: the \emph{reaction kinetics model} and  the \emph{protein network topology}. 
The structure of $r$ has drawn great interest because it affects the blow up property of \eqref{main_eq} \cite{Pierre2000}. 
Therefore, we need to answer the following question:
 {\em `is there  a general reaction kinetics model that guarantees that the reaction-diffusion system does not blow up for any arbitrary network topology?'}.
As an initial step to answering this question, we suggest the following assumptions with respect to the reaction kinetics among proteins $1, \cdots, N$: 
\begin{enumerate}[(A)]
\item \label{a}
No more than two protein molecules can bind to each other at one time;
\item \label{b}
Two protein molecules at most are generated by the dissociation of a complex;
\item \label{c}
Binding and dissociation cannot occur at the same time. 
\end{enumerate}
The reaction kinetics model that we propose is a \emph{mass-action kinetics model that satisfies assumptions (\ref{a}), (\ref{b}) and (\ref{c}).}
For example, consider the protein network depicted in Figure \ref{fig:three}: Protein $\bold{A}$ phosphorylates protein $\bold{B}$, $\bold{B}$ phosphorylates protein $\bold{C}$, and $\bold{C}$ dephosphorylates $\bold{A}$.
The chemical kinetics of the (de)phosphorylations can be modeled as
\begin{equation} \label{kinetic}
\begin{split}
&p\bold{A}+\bold{B} \overset{k_1}{\underset{k_2} {\rightleftharpoons}} p \bold{AB}, \quad 
p\bold{AB} \overset{k_3}{\underset{k_4} {\rightleftharpoons}} p\bold{A} + p\bold{B}, \\
&p\bold{B}+\bold{C} \overset{k_5}{\underset{k_6} {\rightleftharpoons}} p \bold{BC}, \quad p\bold{BC} \overset{k_7}{\underset{k_8} {\rightleftharpoons}} p\bold{B} + p\bold{C}, \\
&p\bold{C}+p\bold{A} \overset{k_9}{\underset{k_{10}} {\rightleftharpoons}} p \bold{CA}, \quad 
p\bold{CA} \overset{k_{11}}{\underset{k_{12}} {\rightleftharpoons}} p\bold{C} + \bold{A},
\end{split}
\end{equation}
where $p\bold{M}$ denotes the phosphorylated $\bold{M}$.
\begin{figure}[tp]
\begin{center}
\includegraphics[width = 2.5in]{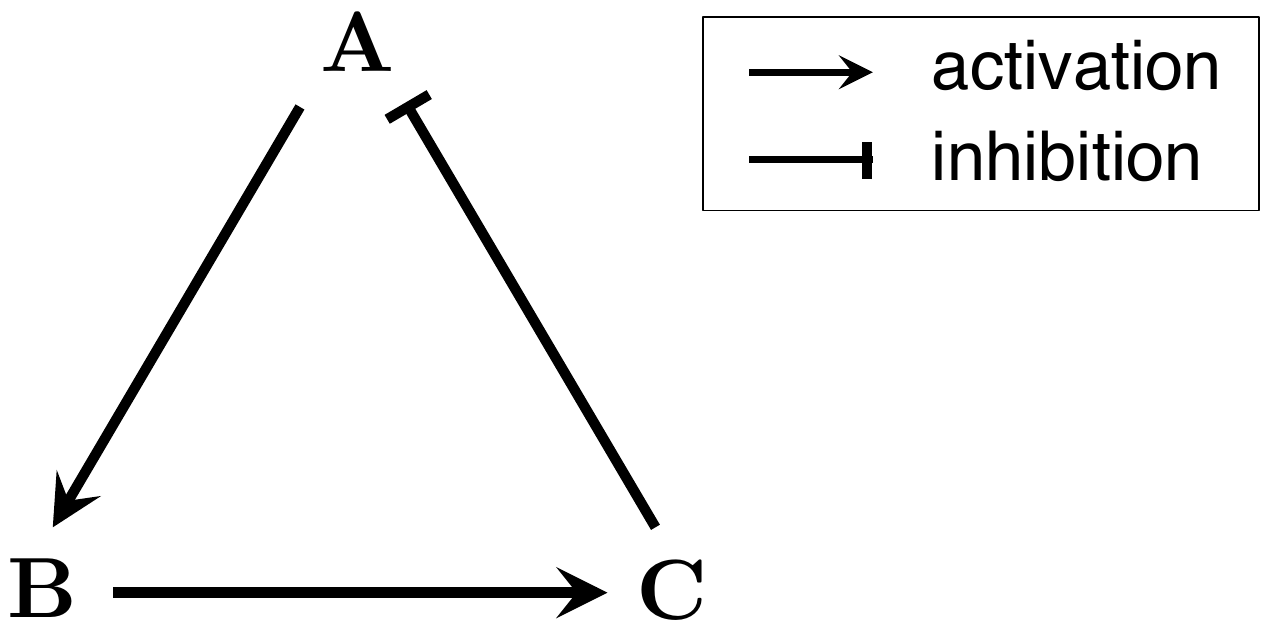}
\caption{A simple protein network.
\label{fig:three}}
\end{center}
\end{figure}
If we let $u_1, u_2, u_3, u_4, u_5, u_6, u_7, u_8$, and $u_9$ denote the concentration levels of $p\bold{A}, \bold{A}, p\bold{B}, \bold{B}, p\bold{C}, \bold{C}, p\bold{AB}, p\bold{BC}$, and $p\bold{CA}$, respectively, then the reaction functions that describe \eqref{kinetic} with mass-action kinetics are given by
\begin{equation}  \label{ex_reaction}
\begin{split}
&r_1= -k_1 u_1 u_4 + k_2 u_7 + k_3 u_7 - k_4 u_1 u_3 -k_9 u_1 u_5 + k_{10} u_9 \\
&r_2 = k_{11} u_9 - k_{12} u_2 u_5 \\
&r_3 = k_3 u_7 - k_4 u_1 u_3 - k_5 u_3 u_6 + k_6 u_8 + k_7 u_8 - k_8 u_3 u_5 \\
&r_4 = -k_1 u_1 u_4 + k_2 u_7 \\
&r_5 = k_7 u_8 - k_8 u_3 u_5 -k_9 u_1 u_5 + k_{10} u_9 + k_{11} u_9 - k_{12} u_2 u_5 \\
&r_6 = -k_5 u_3 u_6 + k_6 u_8 \\
&r_7 = k_1 u_1 u_4 - k_2 u_7 - k_3 u_7 + k_4 u_1 u_3\\
&r_8 = k_5 x_3 x_6 - k_6 x_8 - k_7 x_8 + k_8 x_3 x_5\\
&r_9 = k_9 x_1 x_5 - k_{10} x_9 - k_{11} x_9 + k_{12} x_2 x_5
\end{split}
\end{equation}
Note that the chemical equations \eqref{kinetic} satisfy assumptions (\ref{a}), (\ref{b}) and (\ref{c}).
These assumptions are not restrictive: they only require that the reaction-diffusion system describe the dynamics of chemical signals in detail to some degree, for example, these assumptions do not allow simplified dynamics such as the composition of more than two protein molecules (due to (\ref{a})) or the dissociation into multiple protein molecules (due to (\ref{b})).
Importantly, these assumptions are independent of the protein network structure; therefore, they do not rule out any network topologies.
These assumptions play an important role in proving our key result, the global existence of the classical solution of \eqref{main_eq} with the proposed reaction kinetics model. 
Before we present the key result, we categorize the proteins as follows: 
\begin{itemize}
\item $\bold{Cat}_1$ := \{a single protein species\}.
\item $\bold{Cat}_\alpha$ := \{a complex of $\alpha$ species\}, $\alpha=2,3,\cdots$.
\end{itemize}
By definition, any chemical kinetics can generate protein molecules only  within these categories. 
We assume that $i \leq j$ whenever protein $i$ is in $\bold{Cat}_\alpha$ and protein $j$ is in $\bold{Cat}_\beta$ with $\alpha \leq \beta$, by permuting $\{ \mbox{protein $i$}\}_{i=1}^N$ if necessary.

\subsection{The global existence of the classical solution}
We now introduce a definition of the classical solution of \eqref{main_eq} as in \cite{Pierre2010}: 
\begin{definition}
$u$ is said to be the \emph{classical solution} of \eqref{main_eq} provided that
\begin{itemize}
\item $u$ solves \eqref{main_eq} almost everywhere.
\item $u \in C([0,T); L^1(\Omega)^N) \cap L^{\infty}(\Omega \times [0,T-\tau])^N$, for all $\tau \in (0,T)$.
\item $u_t, u_{x_i}, u_{x_i x_j} \in L^p (\Omega \times (\epsilon_1, T-\epsilon_2))^N$ for all $p\in [1, \infty)$, for all $i, j = 1,\cdots,m$ and for all $\epsilon_1, \epsilon_2 \in (0,T)$.
\end{itemize}
If $T$ can be chosen as $+\infty$, we say that the reaction-diffusion system \eqref{main_eq} has a \emph{global classical solution}. 
\end{definition}
\noindent Note that $\tau$ and $\epsilon_2$ can be selected as $0$ when a global classical solution exists.

To prove that the reaction-diffusion system with the proposed reaction kinetics model has a global classical solution regardless of the network topology, we first observe the following features of $r$.

\begin{lemma}\label{lem}
Suppose that the reaction kinetics model is a mass-action kinetics model satisfying the assumptions (\ref{a}), (\ref{b}) and (\ref{c}). Then the followings hold:
\begin{enumerate}[(I)]
\item \label{A}
The reaction functions, $r_i$, $i = 1,\cdots, N$, are quadratic at most;
\item \label{B}
The quadratic terms of the reaction functions of the proteins in $\bold{Cat}_1$ are non-positive for all $u \in [0, +\infty )^N$;
\item \label{C}
For all $u \in [0, +\infty )^N$, any quadratic term of the reaction function of the proteins in $\bold{Cat}_\alpha$, $\alpha =2,3,\cdots$, can be made to be non-positive by adding a linear combination, with non-negative coefficients, of the reaction functions of the proteins in $\bold{Cat}_\beta$ for $\beta = 1,\cdots, \alpha-1$.
\end{enumerate}
\end{lemma}

\begin{proof}
Observation (\ref{A}) is immediate from assumptions (\ref{a}), (\ref{b}) and (\ref{c}). 

To show that (\ref{B}), we observe by assumption (\ref{c}) that 
 a protein in $\bold{Cat}_1$ is generated by either the decomposition of a complex in $\bold{Cat}_\alpha$, $\alpha =2,3,\cdots$ or a conversion of another protein in $\bold{Cat}_1$. Therefore, all of the positive terms of the reaction function are linear at most in $u$. (See, for example, $r_1, r_2, r_3, r_4, r_5$, and $r_6$ in \eqref{ex_reaction}.) 
 
To observe (\ref{C}), we first note that if the composition of two proteins, say $l$ and $m$, in $\bold{Cat}_\beta$ and $\bold{Cat}_\gamma$, respectively, generates a complex, say $n$, in $\bold{Cat}_\alpha$, then $\alpha = \beta+\gamma$ due to (\ref{c}), which implies that $1 \leq \beta, \gamma \leq \alpha-1$. 
In addition, the positive quadratic term in $r_n$ that represents the rate of this composition (e.g. $k_1 u_1 u_4$ in $r_7$) is negative in the quadratic terms in $r_l$ and $r_m$ that represent the depletion rates of the two proteins (e.g. $-k_1 u_1 u_4$ in $r_1$ and $r_4$). Therefore, the positive quadratic term in $r_n$ can be removed by adding a linear combination $\lambda_1 r_l + \lambda_2 r_m$ with $\lambda_1 + \lambda_2 \geq 1$, $\lambda_1, \lambda_2 \geq 0$ (e.g., $r_7 + (0.5 r_1 + 0.5 r_4)$.)  Therefore, an inductive argument enables us to show observation (\ref{C}). 
\end{proof}

Using Lemma \ref{lem}, we show that the proposed reaction kinetic model allows a reaction-diffusion system for any network to have a global classical solution.

\begin{theorem}
Suppose that the reaction kinetics model is a mass-action kinetics model satisfying the assumptions (\ref{a}), (\ref{b}) and (\ref{c}). Then, the reaction-diffusion system has a global classical solution regardless of the protein network topology.
\end{theorem}
\begin{proof}
Based on the three observations (\ref{A}), (\ref{B}) and (\ref{C}), we have an $N\times N$ lower triangular matrix, $L$, with nonnegative elements and positive diagonal elements such that
\begin{equation} \label{diag_bd}
\begin{split}
Lr(u, k) \leq A \left ( 1 + \sum_{i=1}^N u_i \right ) 
\end{split}
\end{equation}
for all $u \in [0, +\infty )^N$, and for some $N$ dimensional vector $A$. The inequality \eqref{diag_bd} holds element-wise  between the two $N$ dimensional vectors. For example, in \eqref{ex_reaction},
\begin{equation} \nonumber
 \def\dsp{\def\baselinestretch{1}\large\normalsize}
 \dsp
L = 
\begin{bmatrix} 
1 & 0 & 0 & 0 & 0 & 0 & 0 & 0 & 0  \\
0 & 1 & 0 & 0 & 0 & 0 & 0 & 0 & 0 \\
0 & 0 & 1 & 0 & 0 & 0 & 0 & 0 & 0 \\
0 & 0 & 0 & 1 & 0 & 0 & 0 & 0 & 0 \\
0 & 0 & 0 & 0 & 1 & 0 & 0 & 0 & 0 \\
0 & 0 & 0 & 0 & 0 & 1 & 0 & 0 & 0 \\
0.5 & 0 & 0.5 & 0.5 & 0 & 0 & 1 & 0 & 0\\
0 & 0 & 0.5 & 0 & 0.5 & 0.5 & 0 & 1 & 0\\
0.5 & 0.5 & 0 & 0 & 0.5 & 0 & 0 & 0 & 1
\end{bmatrix}.
\end{equation}

Next, we show the quasi-positive structure of the reaction functions, meaning that, for each $i=1,\cdots,N$,
\begin{equation} \label{qp}
r_i(u_1,\cdots, u_{i-1},0,u_{i+1},\cdots,u_N,k) \geq 0
\end{equation}
for all $u \in [0,+\infty)^N$. Indeed, the negative contributions to the reaction function of the $i$th protein are only from $(i)$ the composition of itself with others, $(ii)$ the decomposition of itself, and $(iii)$ the conversion to another species. Therefore, if we set $u_i$ to zero, all negative terms are eliminated from the reaction function, which implies the quasi-positivity of $r$. This quasi-positivity also guarantees that $[0, +\infty )^N$ is an invariant set of the classical solution \cite{Lightbourne1977}.

Furthermore we note that, for all $u \in [0, +\infty)^N$,
\begin{equation} \label{sum_bd}
\begin{split}
\sum_{i=1}^N r_i(u,k) \leq a \left ( 1 + \sum_{i=1}^N u_i \right ),
\end{split}
\end{equation}
for some constant $a$. This observation can be proven using a similar argument to the one used to prove observation (\ref{C}).

In \cite{Pierre2010}, it is shown that a reaction-diffusion system that satisfies \eqref{diag_bd}, \eqref{qp} and \eqref{sum_bd} has a unique global classical solution. Therefore, \eqref{main_eq} with the reaction kinetics model that is characterized by (\ref{a}), (\ref{b}) and (\ref{c}) has a unique global classical solution  for an arbitrary network topology because it satisfies \eqref{diag_bd}, \eqref{qp} and \eqref{sum_bd} regardless of the protein network structure.
\end{proof}

Inequalities \eqref{qp} and \eqref{sum_bd} regardless of the protein network structure. \eqref{diag_bd} and \eqref{sum_bd} provide uniform control over the growth of the nonlinear reaction function $r$ to prevent the solution from blowing up in finite time. The quasi-positivity \eqref{qp} allows us to consider only the case in which the solution is nonnegative across time
because this assumption enforces the nonnegative invariance of the solution. 
 The nonnegativity and (essential) boundedness of the classical solution are also useful for modeling biochemical processes in which the solution $u$ represents quantities that must be nonnegative and cannot blow up in finite time, such as chemical concentrations and densities. 

The global existence of the classical solution allows us to consider \eqref{main_eq} up to time $T$ for any arbitrary $T>0$, i.e., we are able to set the time interval as $[0,T]$.
Recall that the global classical solution has the following regularity:
$u \in C([0,T]; L^1(\Omega)^N) \cap L^{\infty}(\Omega \times [0,T])^N$ and $u_t, u_{x_i}, u_{x_i x_j} \in L^p (\Omega \times (\epsilon, T))^N$ for all $p\in [1, \infty)$, for each $i, j = 1,\cdots,m$, and for all $\epsilon \in (0,T)$.
 

With \eqref{qp} and \eqref{sum_bd} only we can at most claim the global existence of a weak solution. In other words, without  \eqref{diag_bd}, the solution may blow up in $L^{\infty}$ in finite time \cite{Pierre2000}. Recall that \eqref{diag_bd} holds due to the assumptions (\ref{a}), (\ref{b}) and (\ref{c}), which require a detailed modeling of chemical signal dynamics.
The detailed reaction kinetics model not only provides biological plausibility but also guarantees the existence of a global classical solution of the reaction-diffusion system regardless of the protein network topology.

\section{Identification method}

Let $F$ be an $N^* \times N$ observation matrix with $N^* \leq N$ such that the states $Fu$ can be measured.  
Suppose that the data $c: \Omega \times [0,T] \rightarrow \mathbb{R}^{N^*}$, which is an experimental measurement of $F u$, are given.
If $u_{i_m}$, for some $m=1,\cdots,q$, are not measured, 
then the initial value of these $u_{i_m}$'s are unknown. Let $I_{m}$, $m=1,\cdots,q$, be the unknown initial value of each $u_{i_m}$, and let $I=(I_{1}, \cdots, I_{q})$. Therefore, $I \in L^{\infty}(\Omega)^q$ represents a vector of unknown initial values of the reaction-diffusion system \eqref{main_eq}. By definition, $N^* + q = N$.
Let $G$ be an affine function such that $G(I) = u^0$. For instance, assume that $F u = u_3$ with $N = 3$ and $N^* = 1$. Then, $I = (I_1, I_2) := (u_1^0, u_2^0)$ and 
\begin{equation} \nonumber
\small
G(I) = 
\begin{bmatrix}
1 & 0  \\ 0 & 1 \\ 0 & 0
\end{bmatrix}
I
+
\begin{bmatrix}
0 \\
0 \\
u_3^0
\end{bmatrix}
=
\begin{bmatrix}
u_1^0 \\
u_2^0 \\
u_3^0
\end{bmatrix}
= u^0.
\end{equation}
Our goal is to estimate the parameters and unknown initial values in \eqref{main_eq} such that the solution of \eqref{main_eq} best matches the data $c$.
The space of the parameters and initial values can be defined by $\Theta :=  \{ (d,k, I) \in \mathbb{R}^N \times \mathbb{R}^M \times L^\infty (\Omega)^q  \: | \: d>0, k>0, I(x)>0 \mbox{ for all }x\in \Omega  \}$. 
 To find an optimal set of unknown  parameters and initial values $\theta = (d, k, I) \in \Theta$ that minimizes the $L^2$-norm of the difference between $Fu$ and $c$, we pose the following optimization problem in the function space $\Theta$ with a PDE constraint. 
\begin{subequations} \label{main_opt}
\begin{align}
\min_{\theta \in \Theta} \quad &J(u[\theta]):=\frac{1}{2}  \int_{0}^{T} \int_{\Omega} {\left\Vert Fu(x,t)- c(x,t)  \right\Vert}_2^2 \: dxdt \label{obj}\\
\text{subject to} \quad & \mbox{reaction-diffusion system} \; \eqref{main_eq} \label{constr_sys}\\
&d_i^L \leq d_i \leq d_i^U, \quad i=1,\cdots, N \label{constr_d}\\
&k_i^L \leq k_i \leq k_i^U, \quad i=1,\cdots, M  \label{constr_k} \\
&I_i^L \leq I_i (x) \leq I_i^U, \quad i = 1,\cdots, q, \; x \in \Omega \label{constr_I}
\end{align}
\end{subequations}
where $u[\theta]$ is the solution of \eqref{main_eq} with the set $\theta$ of unknown parameters and initial values.
Here $d_i^U$, $k_i^U$ and $I_i^U$ are the positive upper bounds of $d_i$, $k_i$ and $I_i$, respectively. Similarly $d_i^L$, $k_i^L$ and $I_i^L$ are the positive lower bounds of $d_i$, $k_i$ and $I_i$, respectively. To numerically solve \eqref{main_opt} efficiently, e.g., using quasi-Newton methods, we need the gradients of the cost function $J$ in \eqref{obj} with respect to $d$, $k$ and $I$. However, $d$, $k$ and $I$ are not explicit to $J$; these parameters implicitly affect $J$ through the constraint \eqref{constr_sys}. This implicit dependence makes it difficult to compute the gradients. Furthermore, it is not known whether these gradients are well-defined and bounded. We utilize the adjoint-based method to prove the existence, uniqueness, and boundedness of these gradients by deriving their explicit forms.

\subsection{Adjoint-based gradient computations}

We begin by introducing a system of PDEs. This system of PDEs is called the \emph{adjoint system} of \eqref{main_eq} with respect to the optimization problem \eqref{main_opt}. We will use the solution of this system to prove the Fr\'{e}chet differentiability of $J$ and to obtain the gradients. The adjoint system is defined as
\begin{subequations} \label{adjoint}
\begin{align}
& -\frac{\partial \mu_i}{\partial t} - d_i \Delta \mu_i = w_i(\mu,u,k) \quad \text{in} \;  \Omega \times [0,T) \label{adjoint_pde}\\
&\frac{\partial \mu_i}{\partial \nu} = 0 \quad \text{on} \; \partial \Omega \times [0,T) \label{adjoint_Nbc}\\
&\mu_i(x,T)=0 \quad \text{in} \;  \Omega, \label{adjoint_final}
\end{align}
\end{subequations}
where $w(\mu,u,k):= \frac{\partial r(u,k)}{\partial u}^\top \mu - F^\top(F u- c)$ and $A^\top$ denotes the transpose of a matrix $A$. The linearity of $w_i$ in $\mu$ enables us to prove the global existence of the classical solution of the adjoint system.
\begin{proposition}
Suppose that a reaction-diffusion system \eqref{main_eq} has a global classical solution. Then its adjoint system \eqref{adjoint} with respect to the optimization problem \eqref{main_opt} also has a global classical solution.
\end{proposition}
A detailed proof is contained in Appendix \ref{app1}.
Before studying the differentiability of the cost function, we record the following regularity of a reaction-diffusion system \eqref{main_eq} and its adjoint system \eqref{adjoint}.
\begin{lemma} \label{lem2}
Let $u$ and $\mu$ be the global classical solutions of \eqref{main_eq} and \eqref{adjoint}, respectively. Then,
\begin{equation} \label{L2}
D u_i, D \mu_i  \in L^2(\Omega \times (0,T))^\eta.
\end{equation}
\end{lemma}

\begin{proof}
We multiply \eqref{pde} by $u$ and integrate over $\Omega \times (0,T)$ to have
\begin{equation} \nonumber
\int_0^T \int_\Omega  \frac{\partial u_i}{\partial t} u_i + d_i \| D u_i \|^2 \: dxdt = \int_0^T \int_\Omega r_i(u,k) u_i  \: dxdt,
\end{equation}
where we used the divergence theorem.
Using the fact that 
\begin{equation} \nonumber
\int_0^T \int_\Omega  \frac{\partial u_i}{\partial t} u_i \: dxdt = \int_0^T \int_\Omega  \frac{1}{2} \frac{\partial u_i^2}{\partial t} \: dxdt = \frac{1}{2} \int_\Omega (u_i(x,T)^2 - u_i(x,0)^2) \: dx,
\end{equation}
we deduce
\begin{equation} \nonumber
\begin{split}
\int_0^T \int_\Omega d_i \| D u_i \|^2 \: dxdt &= \int_0^T \int_\Omega r_i(u,k) u_i  \: dxdt - \frac{1}{2}\int_\Omega (u_i(x,T)^2 - u_i(x,0)^2) \: dx.
\end{split}
\end{equation}
Since $u \in L^\infty (\Omega \times [0,T])^N$, we obtain \eqref{L2}. Similarly, we can prove 
\begin{equation} \nonumber
D \mu_i  \in L^2(\Omega \times (0,T))^\eta
\end{equation}
for the solution of the adjoint system \eqref{adjoint}.
\end{proof}

We state the main result on the Fr\'{e}chet differentiability of the cost function $J$ and the analytic expression of its gradients that is associated with the adjoint state. 
\begin{theorem} \label{prop2}
The cost function $J$ is Fr\'{e}chet differentiable and its gradients with respect to $d$, $k$, and $I$ are given by
\begin{subequations} \label{grad}
\begin{align}
\nabla_{d} J = &-\int_0^{T} \int_{\Omega}  S(\Delta u)^\top  \mu \: dxdt \label{grad_d} \\
\nabla_{k} J=&-\int_0^T \int_{\Omega} \frac{\partial r(u,k)}{\partial k}^\top \mu \: dxdt \label{grad_k} \\
\nabla_{I} J = & - \frac{\partial G(I)}{\partial I}^\top \mu (x,0), \label{grad_I}
\end{align}
\end{subequations}
where $S(\Delta u)$ is an $N \times N$ diagonal matrix with the $i$th entry equal to $\Delta u_i$. Furthermore, $\nabla_d J$ and $\nabla_k J$ are bounded, and $\nabla_I J \in L^\infty (\Omega)^q$ for any network topology of protein-protein interactions.
\end{theorem}
To show that \eqref{grad}, we begin by considering the variation of $u$ due to
the variation of the parameters and initial values $\theta + \tilde{\theta} \in \Theta$ such that $\| \tilde{\theta} \|$ is bounded by some constant, where
$\|{\theta} \| := ( \| {d} \|^2 + \| {k} \|^2 + \| {I} \|_{L^2(\Omega)^q}^2 )^{1/2}$. 
Let  $\tilde{u} = u[\theta + \tilde{\theta}] - u[\theta]$. 
Note that  $u[\theta + \tilde{\theta}]$ is the global classical solution of the reaction-diffusion system \eqref{main_eq} with $d+\tilde{d}, k+\tilde{k}$ and $I+\tilde{I}$ because it satisfies assumptions (\ref{a}), (\ref{b}), (\ref{c}), and $u^0 + \tilde{u}^0 = G(I + \tilde{I}) \in L^\infty(\Omega)^N$ with $u^0(x) + \tilde{u}^0(x)> 0$ for all $x\in \Omega$.
Then $\tilde{u}_i$,  $i = 1,\cdots,N$, solve the following equation:
\begin{subequations} \label{main_eq3}
\begin{align}
&\frac{\partial \tilde{u}_i}{\partial t} - d_i \Delta \tilde{u}_i = \tilde{d}_i \Delta u_i  + f_i \quad \text{in} \; \Omega \times (0,T] \label{pde3} \\
&\frac {\partial \tilde{u}_i }{\partial \nu} = 0 \quad \text{on} \quad \partial \Omega \times (0,T] \label{Nbc3} \\
& \tilde{u}_i(x,0) = \frac{\partial G_i(I)}{\partial I} \tilde{I} \quad \text{in} \;  \Omega \times \{t = 0\}, \label{init3}
\end{align}
\end{subequations} 
where $f_i$ is obtained using the Taylor expansion of $r$, i.e.,
\begin{equation} \nonumber
\begin{split}
f_i = \frac{\partial r_i(u,k)}{\partial u}  \tilde{u} +  \frac{\partial r_i(u,k)}{\partial k} \tilde{k} +\int_0^1 (1-s) \left ( \tilde{v}^\top D^2 r_i(v+s \tilde{v}) \tilde{v}  \right )\: ds
%
\end{split}
\end{equation}
with $v: = (u,k)$ and $\tilde{v} = (\tilde{u}, \tilde{k})$ due to the mean value theorem.
The first step is to prove the following lemma.
\begin{lemma}\label{lem3}
Let  $\tilde{u} = u[\theta + \tilde{\theta}] - u[\theta]$, where $\theta + \tilde{\theta} \in \Theta$ such that $\| \tilde{\theta} \|$ is bounded by some constant. Then,
\begin{equation}\label{high}
\| \tilde{u}  \|_{L^2(\Omega \times (0,T))^N} \leq K\| \tilde{\theta} \|
\end{equation}
for some constant $K$. 
\end{lemma}
\begin{proof}
Fix $\tau \in (0,T)$ and multiply \eqref{pde3} by $\tilde{u}_i$ and integrate over $\Omega \times (0,\tau)$ for each $i=1,\cdots, N$.
\begin{equation} \nonumber
\begin{split}
\int_0^\tau \int_\Omega \frac{\partial \tilde{u}_i}{\partial t} \tilde{u}_i  + d_i  \|D \tilde{u}_i \|^2 \: dx dt = \int_0^\tau \int_\Omega (\tilde{d}_i \Delta u_i + f_i) \tilde{u}_i \: dx dt,
\end{split}
\end{equation} 
where we used the divergence theorem and the Neumann boundary condition \eqref{Nbc3} to obtain $\int_\Omega \tilde{u}_i  \Delta \tilde{u}_i \: dx = - \int_\Omega \| D \tilde{u}_i \|^2 \: dx$. We can rewrite it as
\begin{equation} \nonumber
\begin{split}
&\frac{1}{2}\int_\Omega \tilde{u}_i(x,\tau)^2 \: dx + d_i \int_0^\tau \int_\Omega \|D \tilde{u}_i \|^2 \: dx dt =  \frac{1}{2}\int_\Omega \tilde{u}_i(x,0)^2 \: dx + \int_0^\tau \int_\Omega \tilde{d}_i  \tilde{u}_i \Delta u_i +f_i \tilde{u}_i \: dx dt.
\end{split}
\end{equation}
Note that 
\begin{equation}\nonumber
\int_\Omega \tilde{d}_i \tilde{u}_i \Delta u_i \: dx = - \int_\Omega D  \tilde{u}_i \cdot (\tilde{d}_i D u_i ) \: dx \leq  \int_\Omega \frac{\epsilon}{2} \| D \tilde{u}_i \|^2 + \frac{1}{2\epsilon} \| \tilde{d}_i D u_i \|^2 \: dx
\end{equation}
for $\epsilon > 0$ due to the Schwarz inequality. Let us choose $\epsilon$ so that $ \epsilon < 2d_i$, then  
\begin{equation} \label{good}
\begin{split}
\int_\Omega \tilde{u}_i(x,\tau)^2 \: dx &\leq  \int_\Omega \tilde{u}_i(x,0)^2 \: dx + \int_0^\tau \int_\Omega \frac{\tilde{d}_i^2}{\epsilon} \| D u_i \|^2 +   C_0 (\| \tilde{k} \| ^2 + \| \tilde{u} \|^2) \: dx dt, \\
\end{split}
\end{equation}
for some constant $C_0 (\tilde{\theta})$ because $u, \tilde{u} \in L^\infty (\Omega \times [0,T])^N$ and $\| \tilde{\theta} \|$ is bounded by some constant $K_0$. Let $C_0^* = \sup_{\theta + \tilde{\theta} \in \Theta, \| \tilde{\theta} \| \leq K_0} C_0(\tilde{\theta})$. Summing \eqref{good} over $i = 1,\cdots,N$, we have
\begin{equation} \nonumber
\begin{split}
\int_\Omega \|\tilde{u}(x,\tau) \|^2 \: dx \leq C^* + C \int_0^\tau \int_\Omega \|\tilde{u} \|^2 \: dxdt, 
\end{split}
\end{equation}
where $C^* =\int_\Omega \|\tilde{u}(x,0)\|^2 \: dx +  \int_0^\tau \int_\Omega \sum_{i = 1}^N  \frac{\tilde{d}_i^2}{\epsilon} \| D u_i \|^2 + C \| \tilde{k} \|^2 \: dxdt$ and $C = NC_0^*$. 
We note that $C^* \leq \bar{C} \| \tilde{\theta}\|^2$ for some constant $\bar{C}$ since $D u_i \in L^2( \Omega \times (0,T))^\eta$, where $\eta$ is the dimension of $\Omega$, as shown in Lemma \ref{lem2}.
Using the Gronwall's inequality, we obtain
\begin{equation} \nonumber
\begin{split}
\int_\Omega \|\tilde{u}(x,\tau) \|^2 \: dx \leq C^* (1 + C \tau \exp(C \tau)).
\end{split}
\end{equation}
 Therefore, $\| \tilde{u}  \|_{L^2(\Omega \times (0,T))^N} \leq K \| \tilde{\theta}\|$ for some constant $K$ as desired.
\end{proof}

Lemma \ref{lem3} suggests that $\bold{O}/ \| \tilde{\theta} \|$ tends to $0$ as $\tilde{\theta} \to 0$ in $\Theta$, where 
 $\bold{O} =  O(\| \tilde{u} \|_{L^2(\Omega \times(0,{T}))^N}^2)+ O(\| \tilde{d} \|^2)+O(\| \tilde{k} \|^2) + O(\| \tilde{I} \|_{L^2(\Omega)^q}^2) $ denotes any higher-order terms.
We are ready to show that $J$ is Fr\'{e}chet differentiable by explicitly deriving the Fr\'{e}chet derivative of $J$. 
\begin{proof} (Theorem \ref{prop2})
To obtain the Fr\'{e}chet derivative of $J$, we consider the difference
\begin{equation} \nonumber
\begin{split}
J (u[\theta + \tilde{\theta}]) - J (u[\theta]) &= \int_0^{T}  \int_{\Omega} (F u - c)^\top F \tilde{u} \: dxdt  + \bold{O} \\
&-\int_0^{T} \int_\Omega \mu^\top \left ( -\tilde{u}_t + D \Delta \tilde{u} +\tilde{D} \Delta u + \frac{\partial r(u, k )}{\partial u}\tilde{u} +  \frac{\partial r(u, k)}{\partial k} \tilde{k} + \bold{O} \right )  dxdt\\ 
&=\int_0^{T} \int_{\Omega}  \left ( -\mu_t^\top -\Delta \mu^\top D - w(u,k,\mu)^\top \right ) \tilde{u} \; dxdt\\
&+\left ( \int_0^{T} \int_{\Omega} -\mu^\top S( \Delta u)  \: dxdt  \right) \tilde{d} +\left(  \int_0^{T} \int_{\Omega}  - \mu^\top \frac{\partial r(u, k)}{\partial k} \: dxdt  \right )\tilde{k} \\
&+ \int_{\Omega} -\mu(x,0)^\top \frac{\partial G(I)}{\partial I} \tilde{I}(x) \: dx +\bold{O},
\end{split}
\end{equation}
as $\tilde{\theta} \to 0$ in $\Theta$.
$D$ and $\tilde{D}$ denote the diagonal matrices with $i$th diagonal elements $d_i$ and $\tilde{d}_i$, respectively. In the first equality, we utilized the Taylor expansion of $J$ and added the integral of the PDE \eqref{pde3} over $\Omega \times (0,T)$, which equals zero. In the second inequality, we substituted $\int_0^T \mu(x,t)^\top \tilde{u}_t(x,t) \; dt = - \int_0^T \mu_t(x,t)^\top \tilde{u}(x,t) \; dt + \mu(x,T)^\top \tilde{u}(x,T) - \mu(x,0)^\top \tilde{u}(x,0)$, which is obtained using integration by parts with \eqref{adjoint_final} and \eqref{init3}. 
The first integral of the second equality equals zero due to the adjoint equation \eqref{adjoint_pde}. 
If $\tilde{d}$ and $\tilde{k}$ are fixed as zero, then $\bold{O} /  {\| \tilde{I} \|_{L^\infty(\Omega)^q}}  \to 0$ as $\tilde{I} \to 0$ in $L^\infty(\Omega)^q$.
Thus, $J$ is Fr\'{e}chet differentiable with respect to $I$, i.e.,
\begin{equation} \nonumber
\begin{split}
\frac{\| J(u[d, k, I + \tilde{I}]) - J(u[d,k,I]) -D_I J(I:\tilde{I})  \|}{ {\| \tilde{I} \|_{L^\infty(\Omega)^q}}} \to 0
\end{split}
\end{equation} 
as $\tilde{I} \to 0$ in $ L^\infty(\Omega)^q$, where 
\begin{equation} \nonumber
\begin{split}
D_I J(I:\tilde{I}) = \int_{\Omega} -\mu(x,0)^\top \frac{\partial G(I)}{\partial I} \tilde{I}(x) \; dx
\end{split}
\end{equation}
is the Fr\'{e}chet derivative at $I$ by definition. Therefore the gradient $\nabla_I J$ exists and is given as \eqref{grad_I}. Similarly, we can obtain $D_d J(d : \tilde{d})$ and $D_k J(k : \tilde{k})$ and derive $\nabla_d J$ and $\nabla_k J$ as \eqref{grad_d} and \eqref{grad_k}, respectively. 
The uniqueness of the gradients follows from the uniqueness of the Fr\'{e}chet derivatives.

We also need to check the boundedness of the gradients for any network topology of protein-protein interactions.
Using the divergence theorem, we obtain $\int_\Omega \mu_i \Delta u_i \: dx = -\int_\Omega D \mu_i \cdot D u_i \: dx \leq \frac{1}{2} \int_\Omega \| D \mu_i\|^2 + \| D u_i\|^2 \: dx$.
Therefore, the regularity results $D \mu_i, D u_i \in L^2(\Omega \times (0,T))^\eta$, $i = 1,\cdots,N$, by Lemma \ref{lem2}, yield that $\nabla_d J$ is bounded.
Furthermore,  for any network structure, $\nabla_k J$ is bounded because $r$ is smooth in $u$ and $u \in L^\infty(\Omega \times [0,T])^N$. The essential boundedness of $\mu$ also implies that of $\nabla_I J = - \frac{\partial G(I)}{\partial I} \mu(x,0)$.
\end{proof}
A numerical evaluation of $\nabla_I J$ might be impossible if $\mu(\: \cdot \:, 0)$ blew up. 
The global existence of a classical solution of the reaction-diffusion system \eqref{main_eq} allows the adjoint system \eqref{adjoint} to have a global classical solution, which in turn guarantees the boundedness of $\nabla_d J$, $\nabla_k J$, $\nabla_I J$ for any network topology.
These (essentially) bounded gradients are necessary for the efficient computational optimization approach that we propose below.

\subsection{Numerical Algorithm}

The problem \eqref{main_opt} is a constrained nonlinear optimization problem that is not generally convex. Thus, our goal is to find a locally optimal solution. A local optimum can be efficiently obtained with the appropriate nonlinear programming methods. In general, 
methods that utilize the gradient of the cost function have better convergence properties compared to those that do not use the gradient (e.g., pattern search methods) \cite{Gill1981, Nocedal2006}. 
%
Because we have derived the gradient information $\nabla_d J$, $\nabla_k J$ and $\nabla_I J$, we can use an efficient algorithm to simultaneously optimize both the parameters and initial values in \eqref{main_opt}.
More specifically, the gradients allow us to utilize quasi Newton-type methods with an approximation of the Hessian. These methods have higher convergence rates (super-linear convergence) compared to those that use an alternating variable or coordinate descent, which often converge slowly \cite{Nocedal2006}.

As seen in formulae \eqref{grad}, a numerical evaluation of the gradients  requires the solution values $u$ and $\mu$ of the reaction-diffusion system \eqref{main_eq} and of the adjoint system \eqref{adjoint}, respectively.
Because an analytic form of the solutions $u$ and $\mu$ is not generally available, our proposed algorithm also includes the numerical evaluation of these values. We can stably approximate $u$ and $\mu$ with bounded values on nodes that discretize $\Omega \times [0,T]$ because the reaction-diffusion system and adjoint system that we proposed do not blow up in finite time regardless of the protein network structure. The bounded approximation of $u$ and $\mu$ allows feasible computation of the gradients, particularly $\nabla_I J$, as discussed in the previous section. 
In summary, our numerical algorithm to solve the optimization problem \eqref{main_opt} is as follows:

\begin{enumerate}
\item Initialize $d$, $k$ and $I$.

\item Numerically solve the reaction-diffusion system \eqref{main_eq} to evaluate $u$. 

\item Numerically solve the adjoint system \eqref{adjoint} to evaluate $\mu$.
\item Compute the gradients $\nabla_d J$, $\nabla_k J$ and $\nabla_I J$ using \eqref{grad}.
\item Compute the gradients of the inequality constraints in \eqref{constr_d}, \eqref{constr_k} and \eqref{constr_I}.
\item Update $d$, $k$ and $I$ by computing the descent direction and the step size using the gradients obtained in Steps 4 and 5. Return to Step 2 if the convergence criteria are not satisfied.
\end{enumerate}

In Steps 2 and 3, we use an implicit version of a Cartesian grid finite difference method \cite{Yang2013} to solve the reaction-diffusion system \eqref{main_eq} and its adjoint system \eqref{adjoint} because this scheme easily handles a complex geometry of $\Omega$ using its level set representation. The method also ensures that the numerical solution of \eqref{main_eq} is nonnegative invariant. Alternatively, other methods can be used to compute the numerical solutions \cite{McCorquodale2001, Wolgemuth2010, Johnson1987}.
In Step 6, there are several applicable methods to choose the descent direction and step size for an optimization problem with equality and inequality constraints \cite{Gill1981, Nocedal2006}, including interior-point methods \cite{Byrd1999, Forsgren2002}, sequential quadratic programming (SQP) methods \cite{Boggs1995, Gill2002}, conjugate-gradient methods \cite{Fletcher1964, Polak1969},
and \noindent trust-region \noindent methods \cite{Celis1984, Conn1987}. I employ an interior-point method that is available in MATLAB (Mathworks, Inc., Natick, MA, USA).
This approach combines line search and trust-region methods to compute the descent direction and step size \cite{Waltz2006}. This hybrid algorithm is computationally efficient and robust; if the Hessian of the Lagrangian is locally convex and nonsingular, the algorithm performs the line search, which is computationally efficient; otherwise, the algorithm switches to a conjugate gradient trust-region step \cite{Byrd1999}, which requires no restriction on the Hessian of the Lagrangian. We also utilized the limited memory Broyden--Fletcher--Goldfarb--Shanno (L-BFGS) method \cite{Liu1989} to approximate the inverse of the Hessian matrix. The L-BFGS method significantly reduces the memory that is required to store the $n \times n$ approximate inverse Hessian $H$ by storing a few $n$-dimensional vectors that implicitly represent $H$. If we use $N_h$ nodes to represent the domain $\Omega$, then the dimension of the discretized initial value $I$ is $O(N_h \times q)$. Therefore, $n=O(N_h \times q)$, if we assume that the number of mass diffusivities and rate constants is significantly less than $N_h$. Thus, it may not be feasible to fully store an $H$ of size $O(N_h^2 \times q^2)$. The L-BFGS method resolves this problem of memory usage.

\section{Example: biochemical regulation of f-actin}

EphA2 is a receptor tyrosine kinase that is activated by binding with Ephrin-A1 ligands on the membrane of nearby cells. The EphA2/Ephrin-A1 complexes form microclusters and display inward transport. The radial transport of these complexes has received much attention because it is highly correlated with the tissue invasion potential of malignant tumor cells \cite{Salaita2010}. 
F-actin is an important EphA2 downstream signaling protein (polymer) because it contributes to the determination of a cell's motility, morphology, and other mechanical properties \cite{Pollard2003, Gardel2006}. Furthermore, the actin cytoskeleton is hypothesized to control the geometric distribution of the EphA2/Ephrin-A1 complexes \cite{Salaita2010}. In addition, the downstream signaling pathways of EphA2 involve Rho family GTPases to regulate actin polymerization \cite{Hall1998}. 

\begin{figure}[tp]
\begin{center}
\includegraphics[width = 3.5in]{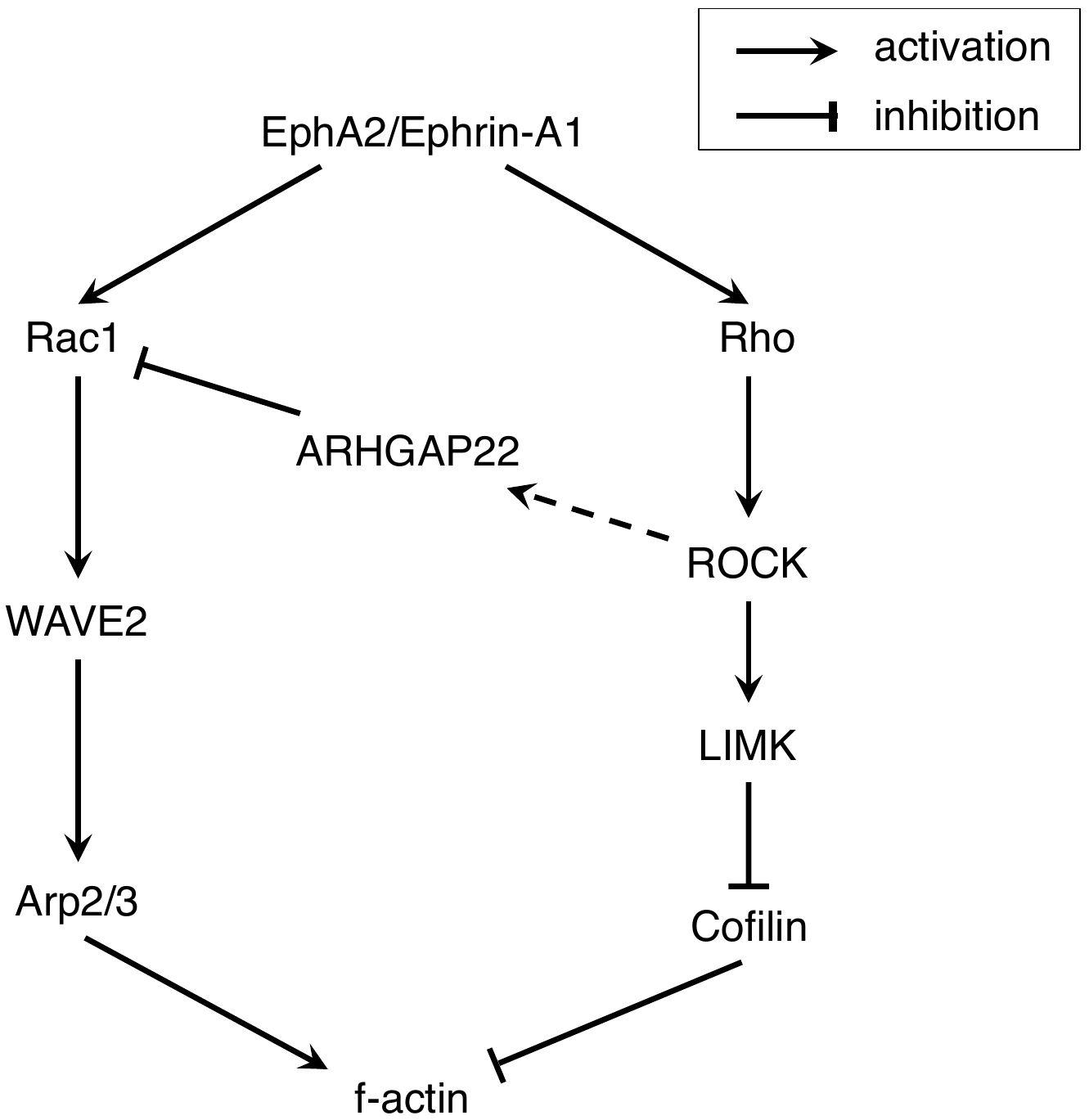}
\caption{EphA2 signaling pathways to f-actin.
\label{fig:pathway}}
\end{center}
\end{figure}

\begin{figure}[tp]
\begin{center}
\includegraphics[width = 4.5in]{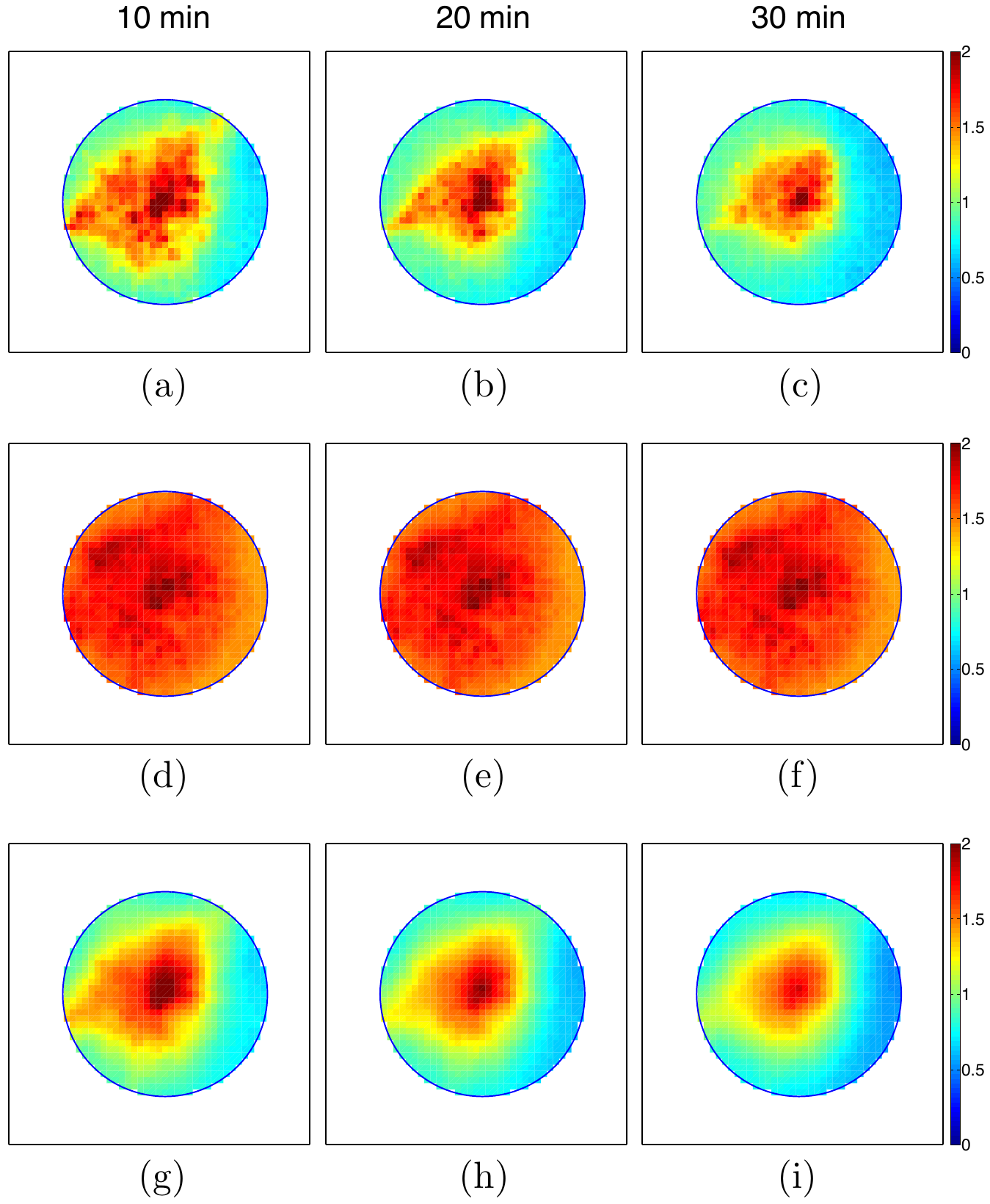}
\caption{Comparison between the data and the solution of the reaction-diffusion model. (a)--(c) f-actin distribution data in an MDA-MB-231 cell (courtesy of Jay T. Groves); (d)--(f)  f-actin distribution predicted by the reaction-diffusion model with random parameters and initial values;  (g)--(h) f-actin distribution predicted by the reaction-diffusion model with estimated parameters and initial values.
\label{fig:actin}}
\end{center}
\end{figure}

In this example, we focus on the signaling network from EphA2/Ephrin-A1 to f-actin to model the biochemical regulation of f-actin through spatio-temporal interactions of important protein species, such as Rho GTPases, Arp2/3 complexes, and Cofilins.
We propose a schematic diagram of the signaling pathways in Figure \ref{fig:pathway} based on the activation-inhibition relations of the protein species (see the SI Text for details). 
Using the chemical kinetics model that we have proposed, we model the spatio-temporal dynamics of the biochemical signals that are transmitted through this protein network with a system of $33$ reaction-diffusion PDEs (see the SI Text for details). The reaction-diffusion system and its corresponding adjoint system have global classical solutions because the reaction kinetics model satisfies assumptions (\ref{a}), (\ref{b}) and (\ref{c}).
Our goal is to identify the mass diffusivities, $d_1, \cdots, d_{33}$, rate constants, $k_1, \cdots, k_{48}$ and initial values, $I_1, \cdots, I_{32}$ of the model given the data on f-actin distribution over the time interval $[0, 10s]$. Note that the initial values $I = (u_1^0, \cdots, u_{32}^0)$ are unknown because only the measurement of $u_{33}$, which represents the concentration of f-actin, is available. Therefore, we set $G(I) = (I, u_{33}^0)$ and $F u = u_{33}$ to compute the cost function as the difference between $u_{33}$ and its measurement. The f-actin distribution data are shown in Figures \ref{fig:actin} (a)--(c). We first randomly chose the parameters and initial values within the 
ranges specified in the SI Text and simulated the reaction-diffusion model with these values. The simulation results (Figures \ref{fig:actin} (d)--(f)) show that the model is unable to capture the contractile redistribution of f-actin that is observed in the data.

We now set this arbitrarily chosen set of parameters and initial values as the initial guess of $d$, $k$ and $I$ of the adjoint-based optimal control algorithm that was proposed in the previous section, and solve the optimization problem \eqref{main_opt} using this algorithm.
The optimized parameters and initial values are shown in Figures \ref{fig:diffusion}, \ref{fig:reaction} and \ref{fig:initial}. The optimized parameters and initial values satisfy the range constraints. 
As shown in Figure \ref{fig:actin} (g)--(i), the reaction-diffusion model with optimal parameters and initial values reproduces the dynamic changes observed in the experimental f-actin distribution. In particular, the model correctly captures the observed contraction of the f-actin distribution over time. Therefore, this model supports the hypothesized interactions of the proteins in the biochemical network in Figure \ref{fig:pathway}. This example suggests that a hypothesized protein network should be tested by an identified model. 
Let us suppose that a reaction-diffusion model of a signaling network with parameters randomly chosen in a physically correct range does not reproduce the data. In this situation, the signaling network may be correct but ruled out due to the inadequate choice of the parameters and initial values. Because the identified parameters and initial values reproduce the data, however, we are not misled to invalidate the potentially correct protein-protein interaction network. Therefore, the identification of parameters and initial values is useful for the validation and invalidation of biological hypotheses as well as the construction of models that are capable of describing certain biochemical processes.

\section{Conclusion}

We have presented a novel reaction kinetics model with which a reaction-diffusion system has a unique global classical solution regardless of the protein network structure.
The classical solution has nonnegative invariance and does not blow up in finite time. These features allow the classical solution to be used to model spatio-temporal signal transductions in a protein network.
We posed an optimization problem 
to determine a set of parameters and initial values with which the reaction-diffusion system best matches the given data. 
By utilizing the adjoint system of this reaction-diffusion systems, we derived the analytic form of the gradients of the cost function with respect to the parameters and initial values. For any network topology, we showed the existence, uniqueness and boundedness of these gradients.
The gradient information allowed us to optimize all parameters and initial values simultaneously using an efficient and robust nonlinear programming method. 
An identified reaction-diffusion system that models the spatio-temporal signaling of f-actin through a hypothesized protein network successfully reproduced the given experimental data.

\section*{Acknowledgments}
The authors would like to thank Professor Lawrence Craig Evans for helpful discussions about the regularity of the classical solution of reaction-diffusion systems. This research was supported by NCI PS-OC program under grant number 29949-31150-44-IQPRJ1- IQCLT.

\appendix
\section{Proof of the global existence of a classical solution to the adjoint system}
\label{app1}

We first derive the following form of the adjoint PDEs by applying a change of variables $\zeta(\cdot,t) = \mu(\cdot,T-t)$ in \eqref{adjoint}: for $i = 1,\cdots,N$,
\begin{equation} \label{adjointS}
\begin{split}
& \partial_t \zeta_i - d_i \Delta \zeta_i = z_i(\zeta,x,t) \quad \text{in} \;  \Omega \times (0,T)\\
&\frac{\partial \zeta_i}{\partial \nu} = 0 \quad \text{on} \; \partial \Omega \times (0,T)\\
&\zeta_i(x,0)=0 \quad \text{in} \;  \Omega,
\end{split}
\end{equation}
where $z_i(\zeta,x,t) := w_i(\mu(x,T-t), u(x,T-t),k)$ for all $x\in \Omega$ and $t\in [0,T]$, for each $i=1,\cdots,N$.
 It is known that there exists $T>0$ such that \eqref{adjointS} has a unique classical solution in $[0,T)$ \cite{Amann1985, Henry1981}. Let $T_{\max}$ denote the greatest of such $T$'s. If we show that $T_{\max} = \infty$, then the global existence is guaranteed. Let $\bold{I} = \bold{I}_1 \times \cdots \times \bold{I}_N$ be a subset of $\mathbb{R}^N$ in which \eqref{adjointS} is invariant.
We propose a Lyapunov function $V \in C^2(\bold{I}; \mathbb{R})$ and $v_i \in C^2(\bold{I}_i; \mathbb{R})$, $i=1,\cdots,N$, satisfying

\begin{enumerate}[(I)]
\item \label{A1}
$V(\xi) = \sum_{i=1}^N v_i(\xi_i)$ for all $ \xi \in \bold{I}$. 
\item \label{A2}
$v_i(\xi_i) \geq 0$, and $v_i^{\prime\prime}(\xi_i) \geq 0$ for all $\xi_i \in \bold{I}_i$, for each $i=1,\cdots,N$.
\item \label{A3}
$V(\xi) \to \infty$ if and only if $\| \xi \|_2 \to \infty$ in $\bold{I}$.
\item \label{A4}
There exists $A \in \mathbb{R}^{N \times N}$ with $A_{ij} \geq 0$ and $A_{ii} > 0$ for $i,j = 1,\cdots,N$ such that for each $j = 1,\cdots, N$, 
$\sum_{i=1}^j A_{ji} v_i^{\prime}(\xi_i)z_i(\xi) \leq C_1(H(\xi))^{p_1} + C_2$ for all $\xi \in \bold{I}$ for some nonnegative constants $p_1$, $C_1$, and $C_2$.
\item \label{A5}
There exist nonnegative constants $p_2, C_3, C_4$ such that $v_i^\prime (\xi_i) z_i(\xi) \leq C_3(V(\xi))^{p_2} + C_4$ for all $\xi \in \bold{I}$, for each $i=1,\cdots, N$.
\item \label{A6}
There exist nonnegative constants $C_5, C_6$ such that $\nabla V(\xi) \cdot z(\xi) \leq C_5 V(\xi) + C_6$ for all $\xi \in \bold{I}$.
\end{enumerate}
Under these assumptions, we have the first result that the Lyapunov function $V(\zeta(\cdot,t))$ is bounded in $L^1$ (see \cite{Morgan1989} Theorem 3.3.):

\begin{theorem} \label{thm1}
Suppose that $V \in C^2(\bold{I}; \mathbb{R})$ satisfies \emph{(\ref{A1})--(\ref{A3}),(\ref{A6})}, and $\zeta$ is a classical solution of \eqref{adjointS} for $[0, T_{\max})$. Then, there exists a function $f \in C([0, \infty))$ such that 
\begin{equation} \nonumber
\| V(\zeta(\cdot,t)) \|_{L^1(\Omega)} \leq f(t)
\end{equation}
for all $t \in (0,T_{\max})$.
\end{theorem}
This $L^1$ boundedness can be used to prove global existence with the following result (see \cite{Morgan1989} Theorem 2.4.):

\begin{theorem} \label{thm2}
Suppose that $V \in C^2(\bold{I}; \mathbb{R})$ satisfy \emph{(\ref{A1})--(\ref{A5})}, and $\zeta$ is a classical solution of \eqref{adjointS} on $[0, T_{\max})$. If there exists a function $g \in C([0, \infty))$ and a positive constant $\alpha$ such that 
\begin{equation} \nonumber
\int_0^t \int_{\Omega} (V(\zeta(x,s))^\alpha \: dxds \leq g(t) \quad \text{for all $t \in (0,T_{\max})$}, 
\end{equation}
and $p_1 < 1+2\alpha/(\eta + 2)$, where $\eta$ is the dimension of $\Omega$, then $T_{\max} = \infty$. 
\end{theorem}

Combining Theorem 1 and Theorem 2 with $\alpha=1$, we now prove the global existence of a classical solution to \eqref{adjointS}. If a Lyapunov function $V$ satisfies (\ref{A1})--(\ref{A6}) with $p_1 = 1$, then Theorem \ref{thm1} and Theorem \ref{thm2} are automatically satisfied; thus, we conclude $T_{\max} = \infty$. Let $\bold{I} = \mathbb{R}^N$, and $V(\zeta) : = \sum_{i=1}^N \zeta_i^2 \in C^2(\bold{I}; \mathbb{R})$. Then, by setting $v_i(\zeta_i) = \zeta_i^2$, (\ref{A1}) and (\ref{A2}) are satisfied. Note that $V(\zeta(x,t))$ itself is the square of the Euclidean norm of $\zeta(x,t)$, we have (\ref{A3}). Recall that $w$ is affine in $\mu$, so $z$ is affine in $\zeta$. 
In addition, $u$ is in $L^\infty(\Omega \times [0,T])^N$, and $r$ is smooth.
Thus, we are able to choose $a_i = (a_{i1}, \cdots, a_{iN}) \in L^{\infty}(\Omega \times [0,T])^N$ and $b_i \in L^{\infty}(\Omega \times [0,T])$ such that $z_i(\zeta(x,t),x,t) = a_i(x,t) \cdot \zeta(x,t) +b_i(x,t)$ for all $x\in \Omega$ and $t\in [0,T]$, for each $i=1,\cdots,N$.
Let $A$ be an $N \times N$ identity matrix, then we have, for each $j=1,\cdots,N$, 
\begin{equation}\nonumber
\begin{split}
&\sum_{i=1}^j A_{ji} v_i^{\prime}(\zeta_i(x,t))z_i(\zeta(x,t),x,t)\\ 
& = v_j^{\prime}(\zeta_j(x,t))z_j(\zeta(x,t),x,t)\\ 
&= 2 \zeta_j(x,t) (a_j(x,t) \cdot \zeta(x,t) + b_j(x,t)) \\
&\leq \sum_{k=1}^N |a_{jk}(x,t)|(\zeta_j(x,t)^2 + \zeta_k(x,t)^2) + \zeta_j(x,t)^2 + b_j(x,t)^2\\
&\leq \left ( 1+ (1+ N) \max_{1 \leq j,k \leq N} \|a_{jk} \|_{L^{\infty}(\Omega \times [0,T])}  \right ) V(\zeta(x,t)) + \max_{1 \leq j\leq N} \|b_j \|_{L^{\infty}(\Omega \times [0,T])}^2.
\end{split}
\end{equation}
In the third inequality, we utilized the Schwartz inequality $2pq \leq p^2 +q^2$ for $p,q \in \mathbb{R}$. In the fourth inequality, we used $\zeta_j(x,t) \leq V(\zeta(x,t))$.
If we set $p_1 = p_2 := 1$, $C_1 = C_3 := 1+ (1+ N) \max_{1 \leq j,k \leq N} \|a_{jk} \|_{L^{\infty}(\Omega \times [0,T])}$ and $C_2 = C_4 := \max_{1 \leq j\leq N} \|b_j \|_{L^{\infty}(\Omega \times [0,T])}^2$, then (\ref{A4})  and (\ref{A5}) are satisfied. Finally, we consider the condition (\ref{A6}).
\begin{equation}\nonumber
\begin{split}
&\nabla V(\zeta(x,t)) \cdot z(x,t,\zeta(x,t)) \\
&= \sum_{i=1}^N 2 \zeta_i(x,t) (a_i(x,t) \cdot \zeta(x,t) + b_i(x,t)) \\
&\leq N \left ( 1+ (1+ N) \max_{1 \leq j,k \leq N} \|a_{jk} \|_{L^{\infty}(\Omega \times [0,T])}  \right ) V(\zeta(x,t)) + N \max_{1 \leq j\leq N} \|b_j \|_{L^{\infty}(\Omega \times [0,T])}^2,
\end{split}
\end{equation} 
due to the same argument in the previous inequality. Letting $C_5 := N[1+ (1+ N) \max_{1 \leq j,k \leq N}$\\$\|a_{jk} \|_{L^{\infty}(\Omega \times [0,T])}]$ and $C_6 :=N \max_{1 \leq j\leq N} \|b_j \|_{L^{\infty}(\Omega \times [0,T])}^2$, we have (\ref{A6}). 
Therefore the Lyapunov function $V$ satisfies (\ref{A1})--(\ref{A6}). Thus, using Theorem \ref{thm1} and Theorem \ref{thm2}, we conclude that there exists a global classical solution to the adjoint system \eqref{adjointS}.


\section{Biochemical signaling from Ephrin-A1 to f-actin} \label{app2}

In this section, we describe the reactions between biochemical species in the f-actin signaling network (Figure \ref{fig:pathway}) in a breast cancer cell.

\begin{enumerate}
\item
{Ephrin-A1 and Rho family GTPases}\\
Ephrin-A1 is a ligand that binds with the EphA2 receptor and then phosphorylates it in an adjacent cell. 
Subsequently, EphA2/Ephrin-A1 activates RhoA via focal adhesion kinase (FAK) \cite{Parri2009}. 
EphA2/Ephrin-A1 also activates Rac1 via the Vav family guanine nucleotide exchange factors \cite{Hunter2006}.
The following chemical kinetic equations model the interaction between EphA2 and the Rho family of GTPases:

\begin{equation} \nonumber
\begin{split}
\mbox{Ephrin-A1} + \mbox{EphA2} &\overset{k_1}{\underset{k_2} {\rightleftharpoons}} p\mbox{EphA2/Ephrin-A1}\\
p\mbox{EphA2/Ephrin-A1} + \mbox{Rho} &\overset{k_3}{\underset{k_4} {\rightleftharpoons}} p\mbox{EphA2/Ephrin-A1/Rho}\\
p\mbox{EphA2/Ephrin-A1/Rho} &\overset{k_5}{\underset{k_6} {\rightleftharpoons}} p\mbox{EphA2/Ephrin-A1} + p\mbox{Rho}\\
\end{split}
\end{equation}
\begin{equation} \nonumber
\begin{split}
p\mbox{EphA2/Ephrin-A1} + \mbox{Rac1} &\overset{k_7}{\underset{k_8} {\rightleftharpoons}} p\mbox{EphA2/Ephrin-A1/Rac1}\\
p\mbox{EphA2/Ephrin-A1/Rac1} &\overset{k_9}{\underset{k_{10}} {\rightleftharpoons}} p\mbox{EphA2/Ephrin-A1} + p\mbox{Rac1}.\\
\end{split}
\end{equation}
Here, we used the notation $\bold{A/B}$ to denote a complex composed of $\bold{A}$ and $\bold{B}$.\\

\item 
{Interplay between Rac1 and Rho signaling}

Rho-associated protein kinase (ROCK) indirectly activates ARHGAP22 through an unknown mechanism \cite{Sanz-Moreno2008}. To model the indirect activation, we introduce an unknown species, U, which is activated by ROCK and activates ARHGAP22. It is also shown that ARHGAP22 inhibits Rac1 \cite{Sanz-Moreno2008}. 

\begin{equation} \nonumber
\begin{split}
p\mbox{ROCK} + \mbox{U} &\overset{k_{11}}{\underset{k_{12}} {\rightleftharpoons}} p\mbox{ROCK/U} \\
p\mbox{ROCK/U} &\overset{k_{13}}{\underset{k_{14}} {\rightleftharpoons}} p\mbox{ROCK} + p\mbox{U} \\
p\mbox{U} + \mbox{ARHGAP22} &\overset{k_{15}}{\underset{k_{16}} {\rightleftharpoons}} p\mbox{U/ARHGAP22} \\
p\mbox{U/ARHGAP22} &\overset{k_{17}}{\underset{k_{18}} {\rightleftharpoons}} p\mbox{U} + p\mbox{ARHGAP22} \\
p\mbox{ARHGAP22} + p\mbox{Rac1} &\overset{k_{19}}{\underset{k_{20}} {\rightleftharpoons}} p\mbox{ARHGAP22/Rac1}\\
p\mbox{ARHGAP22/Rac1} &\overset{k_{21}}{\underset{k_{22}} {\rightleftharpoons}} p\mbox{ARHGAP22} + \mbox{Rac1}. \\
\end{split}
\end{equation}

\item 
{Actin polymerization via a Rac1 signaling pathway} \\

The actin nucleation protein WAVE2 is activated by Rac1 \cite{Sanz-Moreno2008} in a breast cancer cell. Arp2/3 (actin-related protein 2 and 3) complexes are activated by binding with WAVE2 \cite{Yamaguchi2005}. Then, the actin filaments are nucleated and branched by Arp2/3 complexes \cite{Mullins1998}.

\begin{equation} \nonumber
\begin{split}
p\mbox{Rac1} + \mbox{WAVE2} &\overset{k_{23}}{\underset{k_{24}} {\rightleftharpoons}} p\mbox{Rac1/WAVE2} \\
p\mbox{Rac1/WAVE2} &\overset{k_{25}}{\underset{k_{26}} {\rightleftharpoons}} p\mbox{Rac1} + p\mbox{WAVE2} \\
p\mbox{WAVE2} + \mbox{Arp2/3} &\overset{k_{27}}{\underset{k_{28}} {\rightleftharpoons}} p\mbox{WAVE2/Arp2/3} \\
p\mbox{WAVE2/Arp2/3} &\overset{k_{29}}{\underset{k_{30}} {\rightleftharpoons}} p\mbox{WAVE2} + p\mbox{Arp2/3} \\
p\mbox{Arp2/3} + \mbox{Actin}_{\mbox{\tiny off}}  &\overset{k_{31}}{\underset{k_{32}} {\rightleftharpoons}}  \mbox{Actin}_{\tiny \mbox{on}} \\
\end{split}
\end{equation}

\item 
{Actin depolymerization via a Rho signaling pathway} \\

ROCK, a Rho effector, regulates LIM kinase (LIMK) \cite{Ohashi2000, Yoshioka2003}. Then, LIMK phosporylates cofilin to deactivate it \cite{Arber1998, Yoshioka2003}. Active cofilin severs actin filaments to produce free barbed ends. 

\begin{equation} \nonumber
\begin{split}
p\mbox{Rho} + \mbox{ROCK}  &\overset{k_{33}}{\underset{k_{34}} {\rightleftharpoons}}  p\mbox{Rho/ROCK}\\
p\mbox{Rho/ROCK}  &\overset{k_{35}}{\underset{k_{36}} {\rightleftharpoons}}  p\mbox{Rho} + p\mbox{ROCK}\\
p\mbox{ROCK} + \mbox{LIMK} &\overset{k_{37}}{\underset{k_{38}} {\rightleftharpoons}} p\mbox{ROCK/LIMK} \\
p\mbox{ROCK/LIMK} &\overset{k_{39}}{\underset{k_{40}} {\rightleftharpoons}} p\mbox{ROCK} + p\mbox{LIMK} \\
p\mbox{LIMK} + \mbox{Cofilin} &\overset{k_{41}}{\underset{k_{42}} {\rightleftharpoons}} p\mbox{LIMK/Cofilin} \\
p\mbox{LIMK/Cofilin} &\overset{k_{43}}{\underset{k_{44}} {\rightleftharpoons}} p\mbox{LIMK} + p\mbox{Cofilin} \\
\mbox{Cofilin} + \mbox{Actin}_{\mbox{\tiny on}}  &\overset{k_{45}}{\underset{k_{46}} {\rightleftharpoons}}  \mbox{Cofilin/Actin}_{\tiny \mbox{on}} \\
\mbox{Cofilin/Actin}_{\mbox{\tiny on}}  &\overset{k_{47}}{\underset{k_{48}} {\rightleftharpoons}}  \mbox{Cofilin} + \mbox{Actin}_{\tiny \mbox{off}}. \\
\end{split}
\end{equation}
\end{enumerate}

\subsection{Modeling with a reaction-diffusion system} 

We model the spatio-temporal chemical signaling in the EphA2/Ephrin-A1 -- f-actin network as a reaction-diffusion system. 
Let $[\bold{A}]$ denote the concentration level of a species $\bold{A}$. 
We set
\begin{equation} \nonumber
\begin{split}
&u_1 = [\mbox{EphA2}], \; u_2 = [\mbox{Rho}], \; u_3 = [p\mbox{Rho}], u_4 = [\mbox{Rac1}], u_5 = [p\mbox{Rac1}], \; u_6 = [\mbox{ROCK}], \\
& u_7 = [p\mbox{ROCK}], \;  u_8 = [\mbox{U}], \;  u_9 = [p\mbox{U}], \; u_{10} = [\mbox{ARHGAP22}], \;  u_{11} = [p\mbox{ARHGAP22}], \\
& u_{12} = [\mbox{WAVE2}], \;  u_{13} = [p\mbox{WAVE2}], \; u_{14} = [\mbox{Arp2/3}], \; u_{15} = [p\mbox{Arp2/3}],  \; u_{16} =[\mbox{LIMK}],\\
&  u_{17} = [p\mbox{LIMK}], \; u_{18} = [\mbox{Cofilin}], \;  u_{19} = [p\mbox{Cofilin}], \;  u_{20} = [\mbox{Actin}_{\tiny \mbox{off}}], \; u_{21} = [p\mbox{EphA2/Ephrin-A1}], \\  &u_{22} = [p\mbox{ROCK/U}],\;  u_{23} = [p\mbox{U/ARHGAP22}],\;  u_{24} =  [p\mbox{ARHGAP22/Rac1}], \\
&  u_{25} = [p\mbox{Rac1/WAVE2}], \;  u_{26} = [p\mbox{WAVE2/Arp2/3}], \; u_{27} = [p\mbox{Rho/ROCK}], \\
& u_{28} = [p\mbox{ROCK/LIMK}], \;  u_{29} = [p\mbox{LIMK/Cofilin}],\; u_{30} = [\mbox{Cofilin/Actin}_{\tiny \mbox{on}}],\\
& u_{31} = [p\mbox{EphA2/Ephrin-A1/Rho}],\; u_{32} = [p\mbox{EphA2/Ephrin-A1/Rac1}], \; u_{33} = [\mbox{Actin}_{\tiny \mbox{on}}].
\end{split}
\end{equation}
The data of $v = [\mbox{Ephrin-A1}]$ is given (Figure \ref{fig:Ephrin}).
Then, we can write reaction functions $r_1, \cdots, r_{33}$ as follows: 
\begin{equation}\nonumber
\begin{split}
r_1(u,k) &= -k_1 v u_1 + k_2 u_{21}\\
r_2(u,k) &= -k_3 u_{21} u_2 + k_4 u_{31}\\ 
r_3(u,k) &= k_5 u_{31} - k_6 u_{21} u_3 - k_{33} u_3 u_{6} + k_{34} u_{27} + k_{35} u_{27} - k_{36} u_3 u_{7} \\
r_4(u,k) &= -k_7 u_{21} u_4 + k_8 u_{32} + k_{21} u_{24} - k_{22} u_{11} u_4\\
r_5(u,k) &= k_9 u_{32} - k_{10} u_{21} u_5  - k_{19} u_5 u_{11} + k_{20} u_{24} - k_{23} u_{5} u_{12} + k_{24} u_{25}  + k_{25} u_{25}  - k_{26} u_{5} u_{13} \\
r_6(u,k) &= -k_{33} u_3 u_6  + k_{34} u_{27} \\
r_7(u,k) &= -k_{11} u_{7} u_8  + k_{12} u_{22} + k_{13} u_{22} - k_{14} u_7 u_{9}+ k_{35} u_{27}\\
&  - k_{36} u_3 u_7 - k_{37} u_7 u_{16} + k_{38} u_{28} + k_{39} u_{28} - k_{40} u_7 u_{17} \\
r_8(u,k) &= -k_{11} u_7 u_8 + k_{12} u_{22} \\
r_9(u,k) &= k_{13} u_{22} - k_{14} u_7 u_9 - k_{15} u_9 u_{10} + k_{16} u_{23} + k_{17} u_{23} - k_{18} u_9 u_{11} \\
r_{10}(u,k) &= -k_{15} u_9 u_{10} + k_{16} u_{23} \\
r_{11}(u,k) &= k_{17} u_{23}  - k_{18} u_9 u_{11} - k_{19} u_{11} u_{5} + k_{20} u_{24} + k_{21} u_{24} - k_{22} u_{11} u_{4} \\
r_{12}(u,k) &= -k_{23} u_{5} u_{12} + k_{24} u_{25}\\
r_{13}(u,k) &= k_{25} u_{25} - k_{26} u_{5} u_{13}  - k_{27} u_{13} u_{14} + k_{28} u_{26} + k_{29} u_{26} - k_{30} u_{13} u_{15}\\
r_{14}(u,k) &= -k_{27} u_{13} u_{14} + k_{28} u_{26} \\
r_{15}(u,k) &= k_{29} u_{26}  - k_{30} u_{13} u_{15} - k_{31} u_{15} u_{20} + k_{32} u_{33} \\
r_{16}(u,k) &= -k_{37} u_7 u_{16}  + k_{38} u_{28}\\
r_{17}(u,k) &= k_{39} u_{28}  - k_{40} u_{7} u_{17} - k_{41} u_{17} u_{18} + k_{42} u_{29} + k_{43} u_{29} - k_{44} u_{17} u_{19} \\
r_{18}(u,k) &= -k_{41} u_{17} u_{18} + k_{42} u_{29} - k_{45} u_{18}u_{33} + k_{46} u_{30}+k_{47} u_{30} - k_{48} u_{18} u_{20}\\
r_{19}(u,k) &= k_{43} u_{29} - k_{44} u_{17} u_{19}\\
r_{20}(u,k) &= -k_{31} u_{15} u_{20} + k_{32} u_{33} + k_{47} u_{30} - k_{48} u_{18} u_{20} \\
\end{split} 
\end{equation}
\begin{equation}\nonumber
\begin{split}
r_{21}(u,k) &= k_1 v u_1 - k_2 u_{21} - k_3 u_{21} u_{2} + k_4 u_{31} + k_5 u_{31}-k_6 u_{21} u_{3} - k_7 u_{21} u_{4}  + k_8 u_{32} +k_9 u_{32} \\
& - k_{10} u_{21} u_{5}\\
r_{22}(u,k) &= k_{11} u_{7} u_8 - k_{12} u_{22} - k_{13} u_{22} + k_{14} u_7 u_9\\ 
r_{23}(u,k) &= k_{15} u_9 u_{10} - k_{16} u_{23}  - k_{17} u_{23}  + k_{18} u_9 u_{11}\\
r_{24}(u,k) &= k_{19} u_{11} u_5 - k_{20} u_{24} - k_{21} u_{24} + k_{22} u_{11} u_4\\
r_{25}(u,k) &= k_{23} u_5 u_{12} - k_{24} u_{25}  - k_{25} u_{25} + k_{26} u_5 u_{13}\\
r_{26}(u,k) &= k_{27} u_{13} u_{14}  - k_{28} u_{26} - k_{29} u_{26} + k_{30} u_{13} u_{15}\\
r_{27}(u,k) &= k_{33} u_{3} u_{6}  - k_{34} u_{27} - k_{35} u_{27} + k_{36} u_{3} u_{7}\\
r_{28}(u,k) &= k_{37} u_{7} u_{16} - k_{38} u_{28} - k_{39} u_{28} + k_{40} u_{7} u_{17} \\
r_{29}(u,k) &= k_{41} u_{17}u_{18} - k_{42} u_{29} - k_{43} u_{29}  + k_{44} u_{17} u_{19}\\
r_{30}(u,k) &= k_{45} u_{18} u_{33} - k_{46} u_{30} - k_{47} u_{30} + k_{48} u_{18} u_{20} \\
r_{31}(u,k) &= k_{3} u_{21} u_{2}  - k_{4} u_{31} - k_{5} u_{31}  + k_{6} u_{21}u_{3} \\
r_{32}(u,k) &= k_{7} u_{21} u_{4}  - k_{8} u_{32} - k_{9} u_{32}  + k_{10} u_{21}u_{5} \\
r_{33}(u,k) &= k_{31} u_{15}u_{20} - k_{32} u_{33}   - k_{45} u_{18} u_{33} + k_{46} u_{30}.
\end{split} 
\end{equation}

\subsection{Range of parameters and initial values}

Membrane-bound proteins diffuse more slowly than proteins in the cytosol. 
We choose the lower bound and the upper bound of mass diffusivities as
\begin{equation} \label{d_range1}
d_i^{L} = 0.001 \: \mu m^2 s^{-1}, \quad \mbox{and} \quad d_i^{U} = 0.1 \: \mu m^2 s^{-1},
\end{equation}
for chemical species bound in membrane such as EphA2, active Rho family GTPases  and their complexes, i.e., $i=1,3,5,21,24,25,27,31,32$, and
\begin{equation} \label{d_range2}
d_i^{L} = 0.1 \: \mu m^2 s^{-1}, \quad \mbox{and} \quad d_i^{U} = 1 \: \mu m^2 s^{-1},
\end{equation}
for chemical species in the cytosol. Here, we also assume that the assembled actin diffuses slowly by setting $d_{33} = 10^{-16}\: \mu m^2 s^{-1}$.

We use a non-dimensionalized concentration and the following range of rate constants:
\begin{equation} \label{k_range1}
k_{i}^{L}  = 10^{-3} \: s^{-1}, \quad \mbox{and} \quad k_{i}^{U} = 10 \: s^{-1}.
\end{equation}
for forward reactions, and
\begin{equation} \label{k_range2}
k_{i}^{L}  = 10^{-7} \: s^{-1}, \quad \mbox{and} \quad k_{i}^{U} = 10^{-3} \: s^{-1}.
\end{equation}
for backward reactions. 

The lower and upper bounds for the initial (non-dimensionalized) concentration levels are chosen as
\begin{equation} \label{I_range}
I_{i}^{L} = 10^{-4}, \quad \mbox{and} \quad I_{i}^{U} = 1
\end{equation}
for $i = 1,\cdots,32$.

%


\newpage

\begin{figure*}[h]
\begin{center}
\includegraphics[width = 6in]{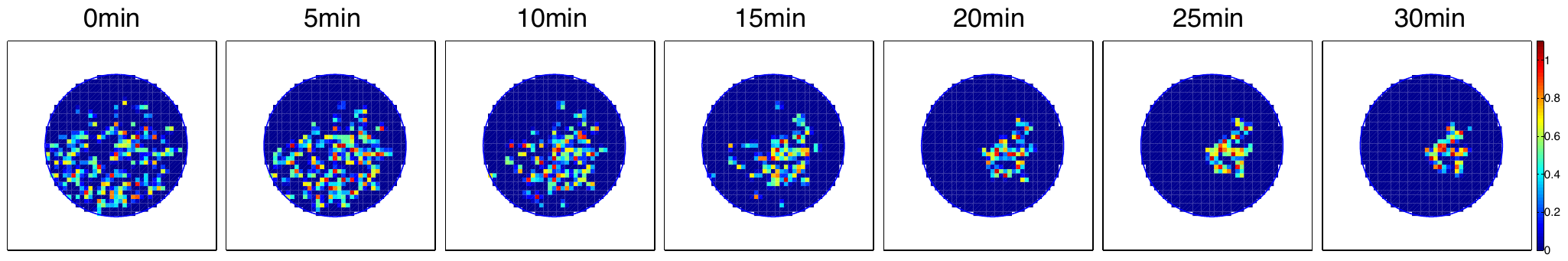}
\caption{Ephrin distribution data in an MDA-MB-231 cell (courtesy of Jay T. Groves.)
\label{fig:Ephrin}}
\end{center}
\end{figure*}

\begin{figure*}[h]
\begin{center}
\includegraphics[width = 6in]{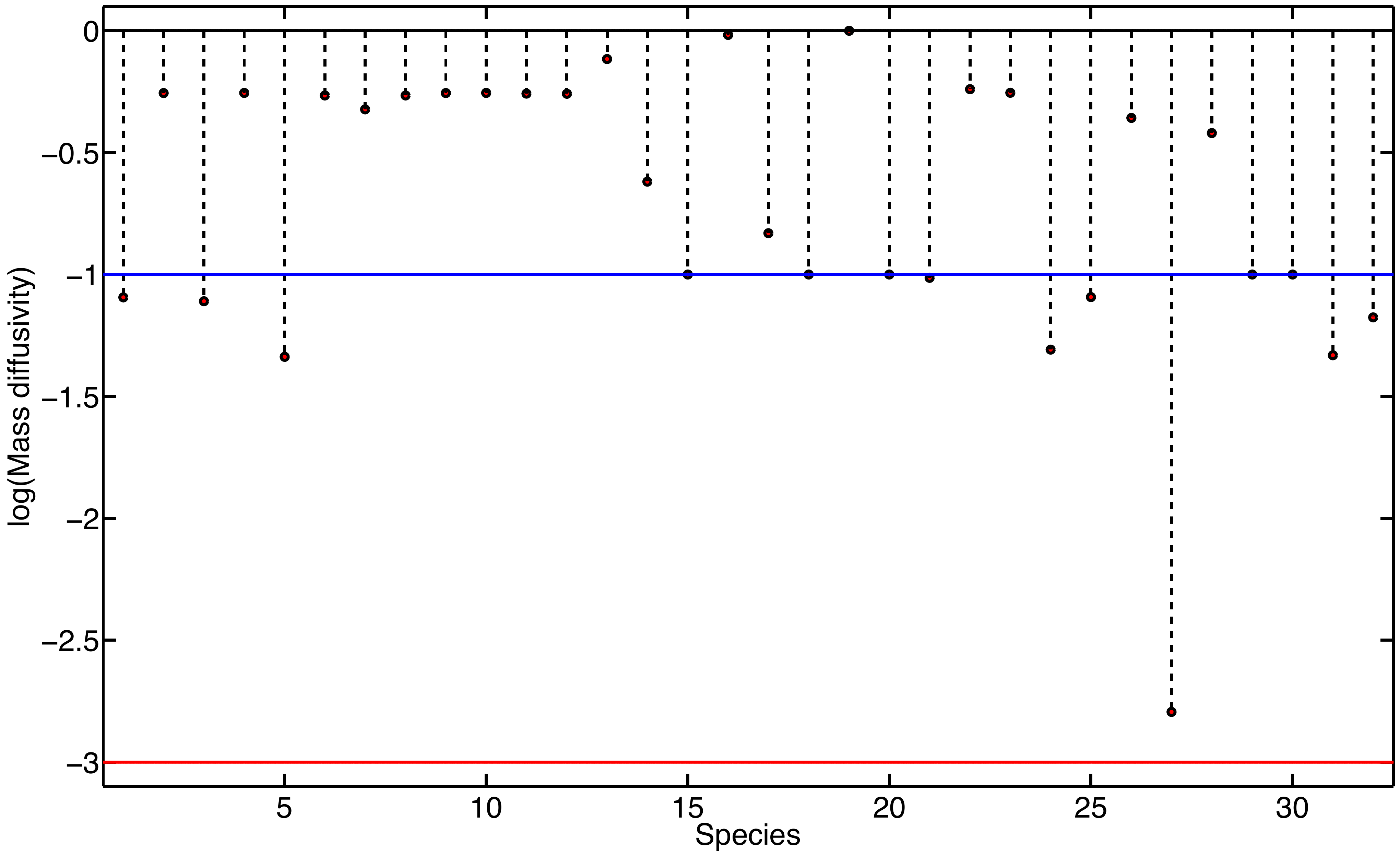}
\caption{Identified $32$ mass diffusivities in log-scale.
The black, blue and red horizontal lines specify the bounds of the mass diffusivities  assumed in \eqref{d_range1} and \eqref{d_range2}. 
Note that $d_1, d_3, d_5, d_{21}, d_{24}, d_{25}, d_{27}, d_{31}, d_{32}$ are within the bounds \eqref{d_range1} and other mass diffusivities are in the bounds \eqref{d_range2}. 
\label{fig:diffusion}}
\end{center}
\end{figure*}

\begin{figure*}[h]
\begin{center}
\includegraphics[width = 6in]{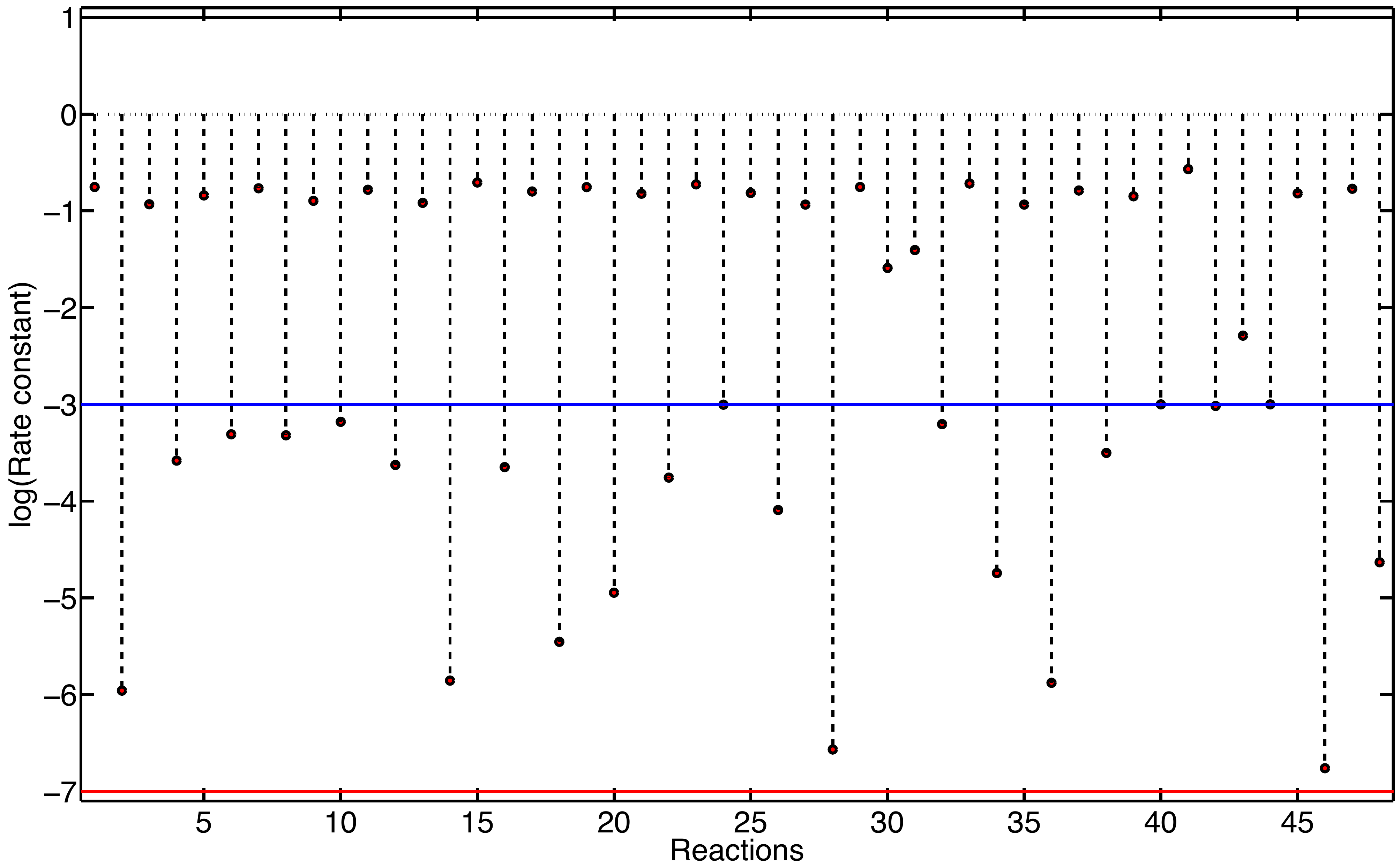}
\caption{Identified $48$ rate constants in log-scale. The black, blue and red horizontal lines specify the bounds of the rate constants  assumed in \eqref{k_range1} and \eqref{k_range2}.
Note that forward and the backward reaction rate constants are within the bounds \eqref{k_range1} and \eqref{k_range2}, respectively.
\label{fig:reaction}}
\end{center}
\end{figure*}

\begin{figure*}[h]
\begin{center}
\includegraphics[width = 6in]{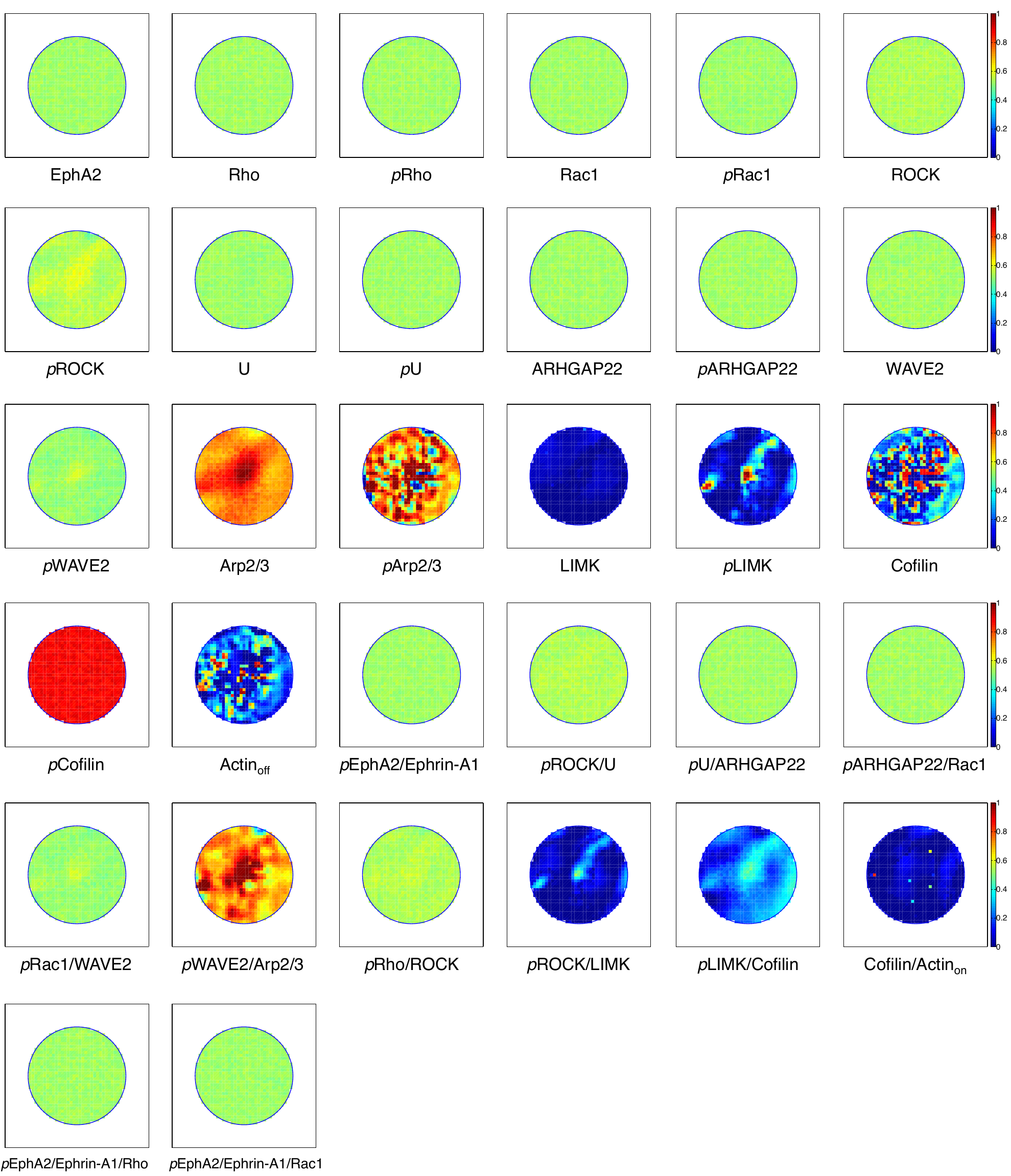}
\caption{Identified $32$ initial values.
Note that the initial values are within the range \eqref{I_range}.
\label{fig:initial}}
\end{center}
\end{figure*}

\clearpage
\bibliographystyle{iEEEtranS}
\bibliography{IDRDNet}

\end{document}